\documentclass[11pt]{amsart}

\usepackage[utf8x]{inputenc} 
\usepackage[a4paper]{geometry} 
\usepackage{enumerate}  
\usepackage{amsmath} 
\usepackage{amssymb}
\usepackage{color}
\usepackage{multirow}
\usepackage{hyperref}  
\usepackage{verbatim}
\usepackage{comment}
                 \hypersetup{ pdfborder={0 0 0}, 
                              colorlinks=true, 
                              linktoc=page,
                              pdfauthor={Michael Hoff and Giovanni Staglian{\`o}}, 
                              pdftitle={Explicit constructions of K3 surfaces and unirational Noether--Lefschetz divisors}} 
\definecolor{Gray}{gray}{0.9}                            
\usepackage[arrow,curve,matrix,tips,frame]{xy}

\theoremstyle{plain} 
\newtheorem{proposition}{Proposition}[section] 
\newtheorem{theorem}[proposition]{Theorem} 
\newtheorem{lemma}[proposition]{Lemma} 
\newtheorem{corollary}[proposition]{Corollary} 
\newtheorem*{maintheorem}{Theorem} 

\theoremstyle{definition} 
\newtheorem{definition}[proposition]{Definition} 
 
\newtheorem{example}[proposition]{Example}

\theoremstyle{remark} 
\newtheorem{remark}[proposition]{Remark}

\renewcommand{\O}{{\mathcal{O}}}   
\newcommand{\QQ}{{\mathbb{Q}}} 

\newcommand{\ZZ}{{\mathbf{Z}}} 
      
\newcommand{\PP}{{\mathbb{P}}}          
\newcommand{\GG}{{\mathbb{G}}}
\newcommand{\cF}{{\mathcal{F}}}

\newcommand{\cH}{{\mathcal{H}}}
\newcommand{\cN}{{\mathcal{N}}}

\newcommand{\YY}{{\mathbb{Y}^5}}

\providecommand{\rk}{\mathop{\rm rk}} 
\providecommand{\Pic}{\mathop{\rm Pic}}

\numberwithin{equation}{section}

\title[Constructions of K3 surfaces and unirational Noether--Lefschetz divisors]{Explicit constructions of K3 surfaces and unirational Noether--Lefschetz divisors}

\author[M. Hoff]{Michael Hoff} 
\address{Universit\"at des Saarlandes, Campus E2 4, D-66123 Saarbr\"ucken, Germany}
\email{\href{mailto:hahn@math.uni-sb.de}{hahn@math.uni-sb.de}} 

\author[G. Staglian\`o]{Giovanni Staglian\`o} 
\address{Dipartimento di Matematica e Informatica, Universit\`{a} degli Studi di Catania, Viale A. Doria 5, 95125 Catania, Italy}
\email{\href{mailto:giovanni.stagliano@unict.it}{giovanni.stagliano@unict.it}} 

\date{\today} 
\thanks{
M.H. was partially supported by the Deutsche Forschungsgemeinschaft (DFG, German Research Foundation) - Project-ID 286237555 - TRR 195. 
}

\begin{document}
\begin{abstract} 
We provide methods to construct explicit examples of $K3$ surfaces. This leads to unirational constructions of Noether--Lefschetz divisors inside the moduli space of $K3$ surfaces of genus $g$. We implement Mukai's unirationality construction of the moduli spaces of $K3$ surfaces of genus $g\in\{6,\dots, 10, 12\}$, and we also present a new constructive proof of the unirationality of the moduli space of $K3$ surfaces of genus $11$. 
Furthermore, we show the existence of three unirational hypersurfaces in any moduli space of $K3$ surfaces of genus $g$.  
\end{abstract}

\maketitle

\section*{Introduction} 

The moduli space $\cF_g$ of (quasi-)polarized $K3$ surfaces of genus $g$ is a fundamental object of interest in algebraic geometry. In order to understand its geometry one has to know/compute its Kodaira dimension. There are many striking results in this direction. On the one hand the unirationality of these moduli spaces is known for small genus (see e.g. \cite{mukaiSugaku, FV18, FV20}) and on the other hand for large genus the space is of general type (see e.g. \cite{GHS07}). We give a complete list of known results in Section~\ref{knownResults}. One says that the moduli space $\cF_g$ is unirational if there exists a dominant rational map $\PP^n\dashrightarrow \cF_g$. 
In this paper, we are more interested in the \emph{explicit unirationality} of $\mathcal{F}_g$,
meaning that the proof of unirationality provides an explicit algorithm which can be implemented 
in a computer algebra program such as \emph{Macaulay2} \cite{macaulay2}; so, in particular,
we shall be able to determine the equations of the general member of $\mathcal{F}_g$.

The natural next step to go is to ask for the Kodaira dimension of lattice-polarized $K3$ surfaces. A main goal of this article is to provide explicit unirational constructions of $K3$ surfaces of Picard rank $2$ (or equivalently,  Noether--Lefschetz divisors inside $\cF_g$). We provide a \emph{Macaulay2} package, named \emph{K3s} \cite{K3sSource}, where we have implemented all our contructions (see Section~\ref{comput}). 

We recall the content of the article and some applications. 
Due to the work of Mukai (see \cite{Muk88,mukai-biregularclassification,MukaiFano3folds}),
we know that the $19$-dimensional moduli space $\mathcal{F}_g$ of (quasi-)polarized $K3$ surfaces of genus $g$ (and degree $2g-2$ given by a quasi-polarization $L$) is explicitly unirational for $g\leq 12$ and $g\neq 11$. 
Using Mukai's descriptions, we will construct such $K3$ surfaces containing a further curve class $C$ of small degree and/or small genus, giving a $K3$ surface with Picard lattice of rank at least $2$ (see also Section~\ref{overview}). Let $S$ be such a $K3$ surface whose Picard lattice contains a (very) ample class of high self intersection given as a linear combination of the quasi-polarization $L$ and $C$, e.g. $H = aL + bC$ for $a,b\in\mathbb{Z}$. In our \emph{Macaulay2} package \emph{K3s}, we provide a function computing the projective model $S\to \PP(H^0(S,H)^*)$. Examples of projective models of $K3$ surfaces of genus $14,22,44$ are presented in our notes. As an application, we find an unirational Noether--Lefschetz divisor in $\cF_{44}$ which is itself of general type (see Subsection~\ref{K3genus44}). Beyond that we show the following result in Section \ref{K3OfLowCliffordIndex}.

\begin{maintheorem}
There are at least three unirational Noether--Lefschetz divisors in the moduli space $\cF_g$ for any $g\ge 3$.
\end{maintheorem}

In a forthcoming work of the first named author with G. Mezzedimi, we prove the existence of a fourth unirational divisors in the moduli space $\cF_g$ for any $g\ge 3$. It is natural to ask for the number $n(g)$ of unirational divisors in $\cF_g$ depending on the genus $g$, for example we construct $7$ unirational divisors in $\cF_{44}$.  We think that the number $n(g)$ is unbounded for $g\to \infty$. But even more could be true.  Are there infinitely many unirational divisors in the moduli space of $K3$ surfaces of genus $g$?

In a work in progress of A. Auel, M. Bolognesi and the first named author, we apply the explicit construction presented in this article to study codimension $2$ subvarieties inside the moduli space of cubic fourfolds and their relation to Noether--Lefschetz divisors of their associated $K3$ surfaces. 

Beside the implementation of Mukai's unirationality constructions, 
our second major result is an explicit construction of $K3$ surfaces of genus $11$. 
Mukai showed in \cite{Mukg11} that $\mathcal{F}_g$ is also unirational for $g=11$, but his proof does not seem to lead in an obvious way to an algorithm. In Section \ref{K3genus11} we fill this gap by showing how to determine equations for a general K3 surface of genus $11$.
\begin{maintheorem}
 The moduli space $\cF_{11}$ of $K3$ surfaces of genus $11$ is explicitly unirational. 
\end{maintheorem}

\subsection*{Acknowledgements}
We would like to thank Sandro Verra for many helpful comments and suggestions.  We also thank Laurent Manivel for explaining us the rational parametrization of the Mukai variety of genus $10$.

\section{Moduli spaces and Mukai models}

\subsection{Moduli spaces of lattice polarized $K3$ surfaces}
Let $(S,L)$ be a quasi-polarized $K3$ surface of genus $g$ (that is, $L$ is pseudo-ample and $L^2 = 2g-2$) and let $\cF_g$ be the $19$-dimensional moduli space of quasi-polarized $K3$ surfaces of genus $g$ (in particular,  the projective model of a quasi-polarized $K3$ surface can have at most isolated singularities). There are countably many Noether--Lefschetz divisors in $\cF_g$ describing $K3$ surfaces of Picard-rank $\ge 2$. Such a divisor is birational to the moduli space 
$\cF^{\Lambda^{d,n}_{g}}$
of lattice-polarized $K3$ surfaces  for some rank $2$ lattice $\Lambda^{d,n}_{g}$ with the intersection matrix 
$$ 
\begin{pmatrix}2g-2 & d  \\ d & n \end{pmatrix}
$$
with respect to a basis $\{h_1,h_2\}$. 
The moduli space $\cF^{\Lambda^{d,n}_{g}}$ of lattice polarized $K3$ surfaces parametrizes pairs $(S,\varphi)$ consisting of a $K3$ surface and a primitive lattice embedding $\varphi: {\Lambda^{d,n}_{g}}\to \Pic(S)$ such that  $\varphi(\Lambda)$ contains the pseudo-ample class. It is a quasi projective irreducible $18\  (= 20 - \rk(\Lambda))$-dimensional variety by \cite{Do96}. There exists a non-empty open subset of $\cF^{\Lambda^{d,n}_{g}}$ such that $L:=\varphi(h_1)$ is pseudo-ample and hence, inducing a quasi-polarization of genus $g$. By abuse of notation, we write $\cF^{\Lambda^{d,n}_{g}}\subset \cF_g$.

\subsection{Birational descriptions of Mukai models}

We denote by $\Sigma_g$ the Mukai model of genus $g\in\{6,7,8,9,10,12\}$,  these are (homogeneous) varieties with a canonical curve section. In the following table we recall their definitions and their rational parametrizations.  We will write $\GG(r,n)$ for the Grassmannian parametrizing $r$-planes in $\PP^n$. 

\begin{remark} 
Note that the varieties $\Sigma_6$ and $\Sigma_{12}$ are not homogeneous varieties. The Fano threefold $\Sigma_{12}$ is a complete intersection with respect to the rank $9$ vector bundle $\Lambda^2\mathcal{S}^{\oplus 3}$ on $\GG(2,6)$, where $\mathcal{S}$ is the dual of the universal subbundle on $\GG(2,6)$. 
\end{remark}

\begin{table}
\begin{tabular}{|c|c| l |c| l |}
\hline
$g$ & $\Sigma_g$ &  & $\dim(\Sigma_g)$ & Parametrization \\ \hline\hline 
$6$ & $\widehat{\GG(1,4)}$ & Cone over the Grassmannian & $7$ & image of $\PP^7$ given by quadrics\\ 
& & & & through the cone $\widehat{\PP^1\times \PP^2}\subset \PP^6\subset\PP^7$ \\ \hline
$7$ & $OG(5,10)$ & Orthogonal Grassmannian & $10$ &  image of $\PP^{10}$ given by quadrics\\ 
 & & & & through $\GG(1,4)\subset \PP^9\subset \PP^{10}$ \\ \hline
$8$ & $\GG(1,5)$ & Grassmannian of line in $\PP^5$ & $8$ & image of $\PP^{8}$ given by quadrics\\ 
 & & & & through $\PP^1\times \PP^3\subset \PP^7\subset \PP^{8}$ \\ \hline
$9$ & $LG(3,6)$ & Langrangian Grassmannian & $6$ & image of $\PP^{6}$ given by cubics singular\\ 
 & & & & along the Veronese surface in $\PP^5 \subset \PP^6$\\ \hline
$10$ & $G_2/B$ & flag variety associated with the & $5$ & image of $\PP^5$ given by special quartics \\ 
& & adjoint representation of $G_2$ & & singular along a twisted cubic in $\PP^3\subset \PP^5$ \\ \hline
$12$ & $X_{22}$ & Prime Fano-threefold of genus $12$ & $3$ & image of a quadric $Q\subset \PP^4$ given by  \\ 
 & & & & quintics singular along a  \\ 
 & & & & rational normal sextic $C\subset Q$ \\ 
 & & & & (see also Section \ref{FanoGenus12}) \\ \hline 
\end{tabular}
\caption{\label{birationalMukaiModels} Parametrizations of Mukai models}
\end{table}

\begin{remark}
The described parametrizations in Table \ref{birationalMukaiModels} are explained in \cite[Thm. 3.8]{zak-tangent} for $g\in \{6,7,8\}$ and are obtained as the inverse of a tangential projection.  The rational parametrization of the Lagrangian Grassmannian $LG(3,6)$ is described in \cite[Section 2.1 and 3.1]{LM02}. The special rational parametrization of $G_2/B\subset \PP^{13}$ is explained in \cite[Section 2.2]{LM02}. 
\end{remark}

\subsection{$K3$ surfaces of small genus as sections of Mukai models}

For our explicit constructions of $K3$ surfaces, we use the fundamental work of Mukai (see \cite{mukaiSugaku} for a summary).
Mukai classified Brill--Noether general $K3$ surfaces of genus $g\in \{6,\dots,10, 12\}$. We recall his classification. 

\begin{definition}
 A quasi-polarized $K3$ surface $(S,L)$ is \emph{Brill--Noether general} if the inequality $h^0(S,M)\cdot h^0(S,N)< h^0(S,L)$ holds for any pair $(M,N)$ of non-trivial line bundles such that $L \cong M\otimes N$.
\end{definition}

\begin{remark}
The condition of being Brill-Noether general is open in the moduli space of $K3$ surfaces of fixed genus. The locus of Brill--Noether general $K3$ surfaces of genus $g\in \{6, \dots, 10,12\}$ is described in \cite[Lemma 1.7]{GLT15}. In these low genera this condition is equivalent to all the smooth curves in the linear system $|L|$ being Brill--Noether general. 
\end{remark}

\begin{theorem}[Mukai]
 A primitively quasi-polarized $K3$ surface $(S,L)$ of genus $6\le g\le 10$ or $12$ is Brill--Noether general if and only if it is birational to a complete intersection with respect to a vector bundle on a space $\Sigma_g$, listed below: 
 \begin{itemize}
  \item [$g=6$:] a quadratic and a codimension $4$ linear section of the cone $\Sigma_6\subset \PP^{10}$ over the Grassmannian $\GG(1,4)$,
  \item [$g=7$:] a codimension $8$ linear section of the orthogonal Grassmannian $\Sigma_7 = OG(5,10)\subset \PP^{15}$, 
  \item [$g=8$:] a codimension $6$ linear section of the Grassmannian $\Sigma_8 = \GG(1,5)\subset \PP^{14}$,
  \item [$g=9$:] a codimension $4$ linear section of the Langrangian Grassmannian $\Sigma_9 = LG(3,6)\subset \PP^{13}$,
  \item [$g=10$:] a codimension $3$ linear section of the flag variety $\Sigma_{10}\subset \PP^{13}$ of dimension five associated with the adjoint representation of $G_2$,
  \item [$g=12$:] a hyperplane section of a prime Fano-threefold $\Sigma_{12}\subset \PP^{13}$ of genus $12$ and degree $22$. 
 \end{itemize}
\end{theorem}

\begin{remark}
For polarized $K3$ surfaces of genus $g$ (whence with a smooth model in $\PP^g$), the strategy of the above theorem is presented in \cite[\S 4]{mukaiSugaku}. The extension of this result for quasi-polarized $K3$ surfaces is explained in \cite[\S 6 and \S 7]{mukaiSugaku}. 
\end{remark}

\begin{remark}
For $g=6$, we emphasize using the cone of the Grassmannian $\GG(1,4)$ in the statement of the theorem since there are several article which state Mukai's result only for the Grassmannian. If the codimension $4$ linear section contains the vertex of the cone, the $K3$ surface is a double cover of a del Pezzo quintic surface branched along a curve of genus $6$. Such $K3$ surfaces were studied in \cite{AK11, HK20} and lie in codimension $4$ inside the moduli space of $K3$ surfaces of genus $6$.
\end{remark}

\section{Overview of our results}\label{overview}

We present explicit unirational construction of the moduli spaces $\cF^{\Lambda^{d,n}_{g}}$ for the lattices in Table \ref{unirationalCases}. 

\begin{table}[h]
\begin{tabular}{|c|c|c|}
\hline
$(i)$ & $\Lambda^{0,-2}_{g} = \begin{pmatrix}2g-2 & 0  \\ 0 & -2 \end{pmatrix}$ & $g\in \{3,\dots,10,12\}$ \\ \hline
$(ii)$ & $\Lambda^{1,-2}_{g} = \begin{pmatrix}2g-2 & 1 \\ 1 & -2 \end{pmatrix}$ & $g\in \{3,\dots,10,12\}$ \\ \hline
$(iii)$ & $\Lambda^{2,-2}_{g} = \begin{pmatrix}2g-2 & 2  \\ 2 & -2 \end{pmatrix}$ & $g\in \{3,\dots,12\}$ \\ \hline
$(iv)$ & $\Lambda^{d,-2}_{g} = \begin{pmatrix}2g-2 & d  \\ d & -2 \end{pmatrix}$ & $g\in \{3,4,5\}$, $3\le d \le 6$ and $(g,d)\in \{(3,7), (3,8), (5,7),(5,8)\}$  \\ \hline
$(v)$ & $\Lambda^{d,0}_{g} = \begin{pmatrix}2g-2 & d  \\ d & 0 \end{pmatrix}$ & any $g$, $3\le d \le 5$ \\ \hline
$(vi)$ & $\Lambda^{d,0}_{g} = \begin{pmatrix}2g-2 & d  \\ d & 0 \end{pmatrix}$ & $(g,d)\in \{(3,6),(3,7),(3,8),(4,6),(4,7),(5,6),\dots,(5,9)\}$ \\ \hline
\end{tabular}
\caption{\label{unirationalCases} All lattices where an explicit unirational construction is implemented in the \emph{Macaulay2} package \emph{K3s}.}
\end{table}

\begin{remark}
\begin{enumerate}
 \item [(a)] From $K3$ surfaces with Picard lattices as in Table \ref{unirationalCases} one can construct infinitely many other $K3$ surfaces with different geometric features.  Indeed, for a $K3$ surface $S\in \cF^{\Lambda^{d,n}_{g}}$ from Table \ref{unirationalCases}, let $H$ and $C$ be the two curves on $S$ inducing the intersection lattice $\Lambda^{d,n}_{g}$. Our \emph{Macaulay2} package also provides a function that embeds $S$ by the linear system $|aH+bC|$ for $(a,b)\in \ZZ^2$ if possible. Therefore, one can construct a $K3$ surface in $\cF^{\Lambda^{d',n'}_{g'}}$ whenever $\Lambda^{d,n}_{g}$ and $\Lambda^{d',n'}_{g'}$ are congruent to each other.  For example,  considering a $K3$ surface of genus $10$ containing a conic (case (iii) in Table \ref{unirationalCases}) is equivalent to a nodal $K3$ surface of genus $11$. 
 \item [(b)] Case (v) Table \ref{unirationalCases} implies that every moduli space $\cF_g$ of $K3$ surfaces of genus $g\ge 3$ contains at least three unirational hypersurfaces (see Section \ref{K3OfLowCliffordIndex}).
\end{enumerate}
\end{remark}

For $K3$ surfaces of genus $3,4$ and $5$ (hence, complete intersections), the unirational constructions are standard (see \cite{FHM20} for details and an application to elliptic $K3$ surfaces). We briefly recall the general strategy in Section \ref{parameterCountGenus345}.

We use Mukai models to construct nodal $K3$ surfaces of genus $g\in \{6,\dots,10, 12\}$. The projection from the node yields a $K3$ surface of genus one less and containing a conic. We explain the unirational construction in Section \ref{K3genus44} for $K3$ surfaces of genus $11$ containing a conic and note that the other cases for $g\le 9$ are similar. All constructions are implemented in our \emph{Macaulay2} package.
Section \ref{FanoGenus12} gives a construction of $K3$ of genus $12$ containing a conic. 

\section{K3 surfaces as complete intersections}
\label{parameterCountGenus345}
\subsection{Hilbert scheme of curves and parameter count: general strategy}
For the rest of this subsection, let $g\in \{3,4,5\}$, $d\ge 1$ and $n = 2k-2 \ge -2$ be integers. 
It is well-known that polarized $K3$ surfaces of genus $g$ are complete intersections in $\PP^g$ of type $(4)$, $(2,3)$ and $(2,2,2)$, respectively. Let $S\in \cF^{\Lambda^{d,n}_{g}}$ be a lattice polarized $K3$ surface with a projective model $S\subset  \PP^g$. Hence, there exists a curve $C\subset S \subset \PP^g$ of degree $d$ and genus $k$. Together with the hyperplane section $h$ of $S$, the intersection matrix with respect to the basis $\{h,[C]\}$ is the lattice $\Lambda^{d,n}_{g}$. We explain the classical strategy to prove the unirationality of $\cF^{\Lambda^{d,n}_{g}}$ which can be reduced to a parameter count and the computation of a single example with the desired properties. 

Let $\cH$ be the Hilbert scheme of curves of degree $d$ and genus $k$ in $\PP^g$. We assume that $\cH$ is unirational (e.g. the Hilbert scheme of rational/elliptic curves in $\PP^g$). We consider the following incidence variety 
 $$I={\{(S, C)\;:\; S\in\cF^{\Lambda_{g}^{d,n}},\; C\in \cH \mbox{ a smooth  curve on } S\}}\subset \cF^{\Lambda_{g}^{d,n}}\times\cH/\text{PGL}(g+1),$$
 together with the two projections
 $$\xymatrix{
 &I\ar[dl]_{p_1}\ar[dr]^{p_2}&&\\
 \cF^{\Lambda_{g}^{d,n}} &&\cH/\text{PGL}(g+1)&}
 $$
Note that the fiber of $p_1$ over a general surface is a $\PP^k (\cong |C|)$ and the fiber of $p_2$ over a curve $C$ is isomorphic to an iterated Grassmannian, and hence rational if non-empty. More precisely, 
in the case $g=3$, the fiber is $|\mathcal{I}_{C}(4)|$, in the case $g=4$, the fiber is isomorphic to a projective bundle $\PP \mathcal{E}$  over $|\mathcal{I}_{C}(2)|$, whose fiber over $q \in |\mathcal{I}_{C}(2)|$ is $\PP \mathcal{E}_q$ defined in the exact sequence
$$0 \rightarrow H^0(\mathcal{O}_{\PP^4}(1)) \stackrel{\cdot q}{\rightarrow} H^0(\mathcal{I}_{C}(3)) \rightarrow \mathcal{E}_q \rightarrow 0.$$
In the remaining case $g=5$, the fiber of $p_2$ is $\GG(2,|\mathcal{I}_{C}(2)|)$.
It follows the unirationality of the incidence variety $I$. Now if we check in a single example of a smooth curve $C$ that $\dim(p_2^{-1}(C))$ is minimal and that
$$
 h^0(C,\cN_{C/\PP^g}) - h^1(C,\cN_{C/\PP^g}) + \dim(p_2^{-1}(C)) = \dim \cF^{\Lambda^{d,n}_{g}} + k ,
$$ 
then we can conclude by semi-continuity that $\dim I =  \dim \cF^{\Lambda^{d,n}_{g}} + k$ is as expected, the first projection $p_1$ is dominant, and hence, the unirationality of $\cF^{\Lambda^{d,n}_{g}}$ follows. 

\begin{remark}
 The Hilbert scheme of rational/elliptic curves in $\PP^g$ for $g\in \{3,4,5\}$ is irreducible and unirational. For the cases in Table \ref{unirationalCases} (iv) and (vi),  the fiber of $p_2$ over a general element is non-empty and therefore, we can apply the above strategy to conclude unirationality.  
\end{remark}

\subsection{Application: K3 surfaces of genus 14}\label{K3genus14}

By \cite{Nuer} and \cite{FV18} the moduli space $\cF_{14}$ of polarized $K3$ surfaces of genus $14$ (and degree $26$) is unirational. The starting point of their unirationality proofs is the following isomorphism. For a general cubic fourfold $X\subset \PP^5$ of discriminant $26$ the Fano scheme of lines on $X$ is isomorphic to the Hilbert square $S^{[2]}$ of a uniquely determined $K3$ surface $S$ of genus $14$. By sending $X$ to $S$, we get a birational isomorphism between the moduli space $\mathcal C_{26}$ 
of cubic fourfolds of discriminant $26$ and the moduli space $\mathcal F_{14}$ of polarized K3 surfaces of genus $14$ (see \cite{Has00}). We will give an explicitly unirational construction of a Noether--Lefschetz divisor in $\cF_{14}$.

We consider the lattice $\Lambda_{5}^{9,0}$ defined by the intersection matrix 
$$ 
\begin{pmatrix}8 & 9  \\ 9 & 0 \end{pmatrix}
$$
with respect to an ordered basis $\{h_1,h_2\}$. We may assume that the class $h_1+h_2$ is big and nef or even ample in general (see \cite[VIII, Prop. 3.10]{BHPV} and note that there are no $(-2)$-classes in $\Lambda_{5}^{9,0}$). Since $(h_1+h_2)^2=26$ and there is a primitive lattice embedding $\langle 26 \rangle\hookrightarrow \Lambda_{5}^{9,0}$, the moduli space $\cF^{\Lambda_{5}^{9,0}}$ is birational to a Noether--Lefschetz divisor in $\cF_{14}$. Therefore we write $\cF^{\Lambda_{5}^{9,0}}\subset \cF_{14}$. 
Applying the above strategy, we get the following proposition. 

\begin{proposition}
 The Noether--Lefschetz divisor $\cF^{\Lambda_{5}^{9,0}}\subset \cF_{14}$ is unirational. 
\end{proposition}
 
\begin{lemma}
 A polarized $K3$ surface $(S,L)\in \cF^{\Lambda_{5}^{9,0}}\subset \cF_{14}$ with Picard lattice $\Lambda_{5}^{9,0}$ is Brill--Noether general. 
\end{lemma}

\begin{proof}
 Let $L,E\in \Pic(S)$ be classes on $S$ such that $L^2 = 26$, $L.E = 9$ and $E^2 = 0$, hence $L$ is a polarization of genus $14$. Let $aL + bE \in \Pic(S)$ be an effective class on $S$. Then $a = \frac{(aL+bE).E}{9}\ge 0$.
 
 Let $M=aL+bE,N=a'L+b'E$ be non-trivial line bundles such that $L = M + N$. 
 Since $a,a'\ge 0$ and $E$ is effective, we may assume that $M = bE$ and $N = L-bE$ with $b > 0$. But $(L-bE)^2 = 26-18b\ge 0$ and therefore, $b=1$. For the only non-trivial decomposition of $L$, the inequality
 $$
 h^0(S,L-E)\cdot h^0(S,E) = 6\cdot 2 < 15 = h^0(S,L)
 $$
 is satisfied.
\end{proof}

\subsection{Application: K3 surfaces of genus 22}\label{K3genus22} 

In \cite{FV20} Farkas and Verra prove the unirationality of the moduli space $\cF_{22}$ of $K3$ surfaces of genus $22$. 
As in the previous section we will give an unirational construction of a Noether--Lefschetz divisor in $\cF_{22}$ where our construction is implemented in \emph{Macaulay2}.

Let $\Lambda_{5}^{3,-2}$ be the rank $2$ lattice defined by the intersection matrix 
$$ 
\begin{pmatrix}8 & 3  \\ 3 & -2 \end{pmatrix}
$$
with respect to an ordered basis $\{h_1,h_2\}$. We may assume that the class $2h_1+h_2$ is big and nef or even ample in general (see \cite[VIII, Prop. 3.10]{BHPV}). Note that $h_2$ is the only $(-2)$-class in $\Lambda_{5}^{3,-2}$. Since $(2h_1+h_2)^2=42$ and there is a primitive lattice embedding $\langle 42 \rangle\hookrightarrow \Lambda_{5}^{3,-2}$, the moduli space $\cF^{\Lambda_{5}^{3,-2}}$ is birational to a Noether--Lefschetz divisor in $\cF_{22}$. Therefore we write $\cF^{\Lambda_{5}^{3,-2}}\subset \cF_{22}$ and get the following proposition. 

\begin{proposition}
 The Noether--Lefschetz divisor $\cF^{\Lambda_{5}^{3,-2}}\subset \cF_{22}$ is unirational. 
\end{proposition}

\begin{remark}
 Let $S\in \cF^{\Lambda_{5}^{3,-2}}$ be a $K3$ surface. Let $H$ and $C$ be the classes in $\Pic(S)$ with $H^2=8$, $H.C=3$ and $C^2=-2$. The polarization of genus $22$ is given by $L = 2H+C$. Then $S\subset \PP^{22}$ is not Brill--Noether general, since $h^0(L)<h^0(H)\cdot h^0(H+C)$. 
\end{remark}

\section{Tri-, tetra- and pentagonal $K3$ surfaces}\label{K3OfLowCliffordIndex}

We briefly recall general facts about $K3$ surfaces contained in a rational normal scroll following \cite{JK04} (see also \cite{Sch}). All results in this section should be already known, but not in the context of unirational hypersurfaces inside moduli spaces of polarized $K3$ surfaces. Similar results for curves can be found in \cite{Sch, Bopp, DP15}.

Let $S\subset \PP^g$ be a polarized $K3$ surface of genus $g$ with polarization $L$ such that $L$ can be decomposed into 
$L \sim E + F$ with $E$ a base point free elliptic pencil (that is, $h^0(S,E) = 2$ and $D\in|E|$ a smooth irreducible elliptic curve) and $h^0(S,F)\ge 2$. 
Then 
$$
X=\bigcup_{D\in |E|\cong \PP^1}\overline{D}\subset \PP^{g}
$$
is a rational normal scroll of dimension $d =h^0(S,L) - h^0(S,L-E)$ and degree $f = h^0(S,L-D)$, where $\overline{D}$ denotes the span of the elliptic curve $D$. 
Conversely, if $S$ is contained in a rational normal scroll $X$ of degree $f\ge 2$, the ruling on $X$ sweeps out a pencil of divisors on $S$ which induces a decomposition of the polarization as above. 

We restrict our considerations to $K3$ surfaces carrying an elliptic pencil of degree $3$ ($4$ and $5$), so-called trigonal (tetragonal and pentagonal, resp.) $K3$ surfaces. Therefore, they are lying on a scroll of dimension $3$ ($4$ and $5$, respectively) and using structure theorems of codimension up to $3$, we have the following proposition. 

\begin{proposition}{\cite[Section 9]{JK04}}
For $3\le d\le 5$, the moduli space $\cF^{\Lambda_g^{d,0}}$ of lattice polarized $K3$ surfaces is unirational. In other words, the moduli space of tri-, tetra- and pentagonal $K3$ surfaces is unirational. 
\end{proposition}

We get immediately the following result for any moduli space of polarized $K3$ surfaces. 

\begin{corollary}
The moduli space $\cF_g$ of $K3$ surfaces of genus $g$ contains at least three unirational Noether--Lefschetz divisors for any genus $g\ge 3$.  
\end{corollary}

In \cite[Section 9]{JK04}, the authors describe all possible families of smooth $K3$ surfaces as in the proposition on rational normal scrolls and compute their number of moduli (see also \cite{SD74} for the trigonal case).  

The proposition follows from a finite number of cases.
Note that after substracting a multiple of the elliptic pencil from the polarization $L$, we may reduce the proof of the proposition to the following cases. Let $(S,L)$ be a polarized $K3$ surface with an elliptic pencil $|E|$ of degree $d\in\{3,4,5\}$ such that 
$$
L^2 = 2g-2 \in 
\begin{cases}
\{4, 6, 8\} & \text{for } d = 3,\\
\{4, 6, 8, 10\} & \text{for } d = 4,\\
\{4, 6, 8, 10, 12\} & \text{for } d = 5.
\end{cases}
$$

The only non-trivial cases, not covered in Section \ref{parameterCountGenus345}, are the following examples. All constructions are unirational since generators/matrices of some bidegrees in the coordinate ring of the rational normal scroll are parametrized linearly. One can compute that the number of moduli is $18$ in all cases. 

\begin{example}{(Trigonal $K3$ surface of genus $5$)}
Let $S\subset \PP^5$ be a smooth $K3$ surface of genus $5$ carrying an elliptic pencil of degree $3$. Then, the surface $S$ lies on a cubic scroll $X$ isomorphic to $\PP^1\times \PP^2\subset \PP^5$ and is cut out by a form of bidegree $(2,3)$ on  $X$.
\end{example}

\begin{example}{(Tetragonal $K3$ surface of genus $6$)}
Let $S\subset \PP^6$ be a smooth $K3$ surface of genus $6$ carrying an elliptic pencil of degree $4$. Then, the surface $S$ lies on a cubic scroll $X$ isomorphic to the cone of $\PP^1\times \PP^2\subset \PP^6$ and is cut out by two forms of bidegree $(1,2)$ and $(2,2)$ on $X$.
\end{example}

\begin{example}{(Pentagonal $K3$ surface of genus $6$)}
Let $S\subset \PP^6$ be a smooth $K3$ surface of genus $6$ carrying an elliptic pencil of degree $5$. Then, the surface $S$ lies on a singular quadric $X$ of rank $4$ and is generated by the $4\times 4$ Pfaffians of $5\times 5$ skew-symmetric matrix on $X$ with homogeneous entries of bidegree $(1,1)$.
\end{example}

\begin{example}{(Pentagonal $K3$ surface of genus $7$)}
Let $S\subset \PP^7$ be a smooth $K3$ surface of genus $7$ carrying an elliptic pencil of degree $5$. Then, the surface $S$ lies on a cubic scroll $X$ isomorphic to the cone of $\PP^1\times \PP^2\subset \PP^7$ and is generated by the $4\times 4$ Pfaffians of $5\times 5$ skew-symmetric matrix on $X$ with homogeneous entries given in the following matrix 
$$
\begin{pmatrix} 
0 & 0 & (1,1) & (1,1) & (1,1) \\
0 & 0 & (1,1) & (1,1)& (1,1) \\
(1,1) & (1,1) & 0 &(1,0) & (1,0)  \\
(1,1) & (1,1) & (1,0) & 0 & (1,0)  \\
(1,1)& (1,1) & (1,0)  & (1,0)  & 0
\end{pmatrix}.
$$
\end{example}

\begin{remark}
 It turns out that the moduli space $\cF^{\Lambda_g^{6,0}}$ of hexagonal $K3$ surfaces is unirational as well. These are $K3$ surfaces of genus $g$ which contain an elliptic pencil of degree $6$. This will be part of a forthcoming paper of the first named author and G. Mezzedimi. 
\end{remark}

\section{Nodal $K3$ surfaces and $K3$ surfaces containing a rational curve}

\subsection{Rational parametrizations of Fano threefolds}

A \emph{prime Fano threefold} $X$ is a smooth projective variety of dimension $3$ whose Picard group is generated by the anti-ample canonical class $K_X$. The genus $g(X)$ is defined by the equality $(-K_X)^3 = 2g(X)-2$. 

We use several methods to parametrize prime Fano threefolds of genus $g\in\{6,\dots,10,12\}$ embedded into $\PP^{g+1}$ together with either a distinguished point or a rational curve on it. There are parametrizations given by the inverse of a double projections from a line for $g\le 10$ (as treated in \cite{Is77, Sho79}) or by a Sarkisov link for $g=12$ (see Section \ref{FanoGenus12}). Note that all such prime Fano threefolds are sections of the corresponding Mukai models (see e.g. \cite{MukaiFano3folds}). In Table \ref{paraFano3folds}, we list our chosen parametrizations (see also Table \ref{birationalMukaiModels}). 

\begin{table}[h]
\begin{tabular}{|c|l|}
\hline
$g$ & Parametrization \\ \hline\hline 
$6$ & section of the Grassmannian  $\GG(1,4)$\\ 
 & containing special Schubert cycles \\ \hline
$7$ &  parametrization of the Mukai model $\Sigma_7$\\ 
  & containing a distinguished point/a rational curve \\ \hline
$8$  & section of the Grassmannian  $\GG(1,5)$\\ 
 & containing special Schubert cycles \\ \hline
$9$ & parametrization of the Mukai model $\Sigma_9$\\ 
  & containing a distinguished point/a rational curve \\ \hline
$10$ & image of a quadric $Q\subset \PP^4$ given by  \\ 
 & quintics singular along a  \\ 
 & a curve $C\subset Q$  of degree $7$ and genus $2$ \\  \hline 
$12$  & image of a quadric $Q\subset \PP^4$ given by  \\ 
 & quintics singular along a  \\ 
 & rational normal sextic $C\subset Q$ \\  \hline 
\end{tabular}
\caption{\label{paraFano3folds} Parametrizations of prime Fano threefolds}
\end{table}

We remark that we get a unirational variety parametrizing pairs of Fano threefolds together with either a point or a rational curve of small degree. Indeed, let $\cH_X(\PP^{g+1})$ be the Hilbert scheme of Fano threefolds $X\subset \PP^{g+1}$ of genus $g$ which is unirational by construction.  If $X$ is rigid, the universal Fano threefold $\mathcal X$ parametrizing pairs of Fano threefolds with a distinguished point is unirational too, since $X$ is rational. In the remaining cases (if $X$ is not rigid), we construct $X$ as a general section of  $\Sigma_g$ containing a general point, and hence, the unirationality of the universal Fano threefold $\mathcal X$ follows. Let $\cH_C(X)$ be the Hilbert scheme of rational curves of low degree on $X$. By our constructions and a similar parameter count as in Section \ref{parameterCountGenus345}, the incidence variety 
$$\{(X,C) | C\subset X\} \subset \cH_X(\PP^{g+1}) \times \cH_C(X)$$ 
is unirational. 

{\bf Nodal $K3$ surfaces.} Having a prime Fano threefold $X$ together with a distinguished point $p$, we can compute the tangent space $T_p X$ to $X$ at $p$. Choosing a general hyperplane section of $X$ containing $T_p X$ yields a $K3$ surface with one simple double point. Such hyperplane sections are parametrized by a projective space. Hence, the moduli space parametrizing such $K3$ surfaces is unirational.  

{\bf $K3$ surfaces containing a rational curve.}
Having a prime Fano threefold $X$ together with a rational curve $C$, let $\overline{C}$ be the linear space of $C$. Choosing a general hyperplane section of $X$ containing $\overline{C}$ yields a $K3$ surface containing $C$. As above, such hyperplane sections are parametrized by a projective space. Hence, the moduli space parametrizing these $K3$ surfaces is unirational.  

We will explain in the following section our procedure for prime Fano threefolds of genus $12$. 

\subsection{Construction of prime Fano threefolds of genus 12}\label{FanoGenus12}

Since the construction of a prime Fano threefold $V_{22}$ of genus $12$ is more subtle, we recall a construction based on a Sarkisov link to a quadric threefold $Q\subset \PP^4$. We follow \cite{KPEpi} (see also \cite{Tak89},\cite{IP99}).
  
The moduli space $\mathcal M_{Fano,12}$ of prime Fano threefolds of genus $12$ has three incarnations (see \cite{mukai-biregularclassification}, \cite{MukaiFano3folds} and \cite{schreyer_2001}). It is birational to the moduli space $\mathcal M_3$ of curves of genus $3$,
to the moduli space of nets of quadrics in $\PP^3$ and to the moduli space of curves of genus $3$ together with a non-vanishing theta characteristic. The birationality relies on geometric realizations of a Fano threefold $V_{22} \in \mathcal M_{Fano,12}$ as a Grassmannian $G(3,V_7,\eta)$ of isotropic $3$-spaces in a $7$-dimensional vector space $V_7$, as the variety of sums of powers presenting the equation of a plane quartic, or as the variety of twisted cubics in $\PP^3$ whose quadric equations are annihilated by a net of quadrics. In particular, the moduli space $\mathcal M_{Fano,12}$ is unirational. But none of these realizations allows one to write down explicit equations of a general Fano threefold $V_{22}$. Using birational geometry we are able to do this. 

Let $V_{22}$ be a general Fano threefolds of genus $12$ such that the anti-canonical divisor $-K_{V_{22}}$ is very ample.  

\begin{proposition}\label{conicV22}
\cite[Thm 0.2]{Tak89}, \cite[Thm 1.1.1]{KPS}
 There exists a smooth conic on $V_{22}$. Furthermore, the Hilbert scheme of conics on $V_{22}$ is isomorphic to $\PP^2$. 
\end{proposition}

\begin{theorem}{\cite[Theorem 2.2, 2.6]{KPEpi}} \label{V22}
Let $C\subset V_{22}$ be a smooth conic on $V_{22}$. There exists the following commutative diagram of birational maps 

\begin{align}\label{SarkisovLinkV22}
\begin{xy}
 \xymatrix{
  & V_{22}' \ar[ld]_{\varphi_C} \ar@{-->}[rr]^{\chi} & & Q' \ar[rd]^{\varphi_\Gamma} & \\
  V_{22}\subset \PP^{13} \ar@{-->}[rrrr]^{\varphi} & & & & Q\subset \PP^4 
 }
\end{xy}
\end{align}

where 
\begin{itemize}
 \item $Q$ is a smooth quadric in $\PP^4$,
 \item $\varphi_C$ is the blow up of $C$,
 \item $\varphi_\Gamma$ is the blow up of a smooth rational (quadratically normal) sextic curve $\Gamma\subset Q$,
\end{itemize}
and $\chi$ is a flop.

The birational map $\varphi$ is given by the hyperplanes on $V_{22}$ double along $C$, that is, the linear system $|\mathcal I_{C/V_{22}}^2(1)|$. The inverse map $\varphi^{-1}$ is given by quintic hypersurfaces double along $\Gamma$, that is, the linear system $|\mathcal I_{\Gamma/Q}^2(5)|$.
Furthermore, the compositions $\varphi_\Gamma \circ \chi$ and $\varphi_C\circ \chi^{-1}$ contract the strict transforms of the unique divisors in $|\mathcal I_{C/V_{22}}^5(2)|$ and $|\mathcal I_{\Gamma/Q}(2)|$, respectively.

To a smooth quadric $Q\subset \PP^4$ containing a rational quadratically normal sextic curve $\Gamma\subset Q$ there exists a smooth Fano threefold $V_{22}$ of genus $12$ and a smooth conic $C\subset V_{22}$ related to $(Q,\Gamma)$ by \ref{SarkisovLinkV22}.
\end{theorem}

Let $\cH$ be the quotient of the Hilbert scheme of plane conics in $\PP^{13}$ under the action of PGL$(14)$ and let $\cH'$ be the quotient of the Hilbert scheme of rational sextic curves $\Gamma\subset \PP^4$ under the action of PGL$(5)$. Note that the general element in $\cH'$ is quadratically normal. 
Let us consider the following incidence correspondences 
$$I_\Gamma=\overline{\{(\Gamma, Q)\;:\; \Gamma \subset  Q\subset \PP^4 \}}\subset \cH'\times \PP(H^0(\O_{\PP^4}(2)))/\text{PGL}(5)$$ 
and 
$$I_C=\overline{\{(C, V_{22})\;:\; C \text{ smooth and }C \subset  V_{22}\subset \PP^{13} \}}\subset \cH\times \mathcal M_{Fano,12}.$$ 

Then $I_\Gamma$ is birational to a $\PP^1$-bundle over $\cH'$ and hence, a $8$-dimensional unirational variety. Similarly, by Proposition \ref{conicV22} $I_C$ is birational to a $\PP^2$-bundle over $\cH$ and thus, a $8$-dimensional unirational variety. Using a relative MMP one can extend Theorem \ref{V22} 
to smooth families. We get the following proposition. 

\begin{proposition}
 The incidence correspondences $I_\Gamma$ and $I_C$ are birational. In particular, there is an explicit unirational method to construct Fano threefolds of genus $12$.
\end{proposition}

\begin{corollary}
The moduli space of $\Lambda_{12}^{2,-2}$-polarized K3 surfaces 
(that is, a polarized K3 surface of genus $12$ containing a conic) is unirational. 
\end{corollary}

\subsection{Application: K3 surfaces of genus 44}\label{K3genus44}

Let $\Lambda_{12}^{0,-2}$ be the rank $2$ lattice defined by the intersection matrix 
$$ 
\begin{pmatrix} 22 & 0 \\ 0 & -2 \end{pmatrix}
$$
with respect to an ordered basis $\{h_1,h_2\}$. Let $\cF^{\Lambda_{12}^{0,-2}}$ be the moduli space of $\Lambda_{12}^{0,-2}$-polarized $K3$ surfaces. 

\begin{remark}
As a Noether--Lefschetz divisors $\cF^{\Lambda_{12}^{0,-2}}\subset \cF_{12}$ it describes exactly those $K3$ surfaces of genus $12$ with rational double points. 
\end{remark}

We may assume that the class $2h_1 - h_2$ is big and nef or even ample in general (see \cite[VIII, Prop. 3.10]{BHPV} and note that $h_2$ is the only $(-2)$-class in $\Lambda_{12}^{0,-2}$). Since $(2h_1-h_2)^2=86$ and there is a primitive lattice embedding $\langle 86 \rangle\hookrightarrow \Lambda_{12}^{0,-2}$, the moduli space $\cF^{\Lambda_{12}^{0,-2}}$ is birational to a Noether--Lefschetz divisor in $\cF_{44}$. Therefore we write $\cF^{\Lambda_{12}^{0,-2}}\subset \cF_{44}$. 

\begin{proposition}
 The Noether--Lefschetz divisor $\cF^{\Lambda_{12}^{0,-2}}\subset \cF_{44}$ is unirational. 
\end{proposition}

\begin{proof}
 Let $\cH$ be the Hilbert scheme of nodal $K3$ surface of genus $12$ in $\PP^{12}$.
 Let us consider the incidence correspondence:
 $$
 \xymatrix{
  I=\{(S, V_{22})\;:\; S \mbox{ nodal and } S \subset V_{22} \}\subset \cH/\text{PGL}(13)\times \mathcal M_{Fano,12} \ar[d]_{p_1}\ar[dr]^{p_2}&\\
 \cH/\text{PGL}(13) & \mathcal M_{Fano,12} }
 $$
 By \cite{mukaiSugaku} 
 the first projection $p_1$ is birational (there exists an unique stable rank $3$ bundle with $7$ global sections which induces the embedding $S\subset V_{22}$). For a general Fano threefold $V_{22}$, the fiber of $p_2$ is unirational. Indeed, 
 $$(p_2)^{-1}(V_{22}) = \{(p, H)\;:\; p\in V_{22}, H\supset T_p(V_{22})\} = V_{22}\times \PP^{9} \cong \PP^{12},$$ 
 since $V_{22}$ is rational. Hence, $I$ and $\cH/\text{PGL}(13)$ is rational. Finally, the quotient $\cH/\text{PGL}(13)$ is birational to $\cF^{\Lambda_{12}^{0,-2}}$ by sending a nodal $K3$ surface to its desingularization.
\end{proof}

\begin{remark}
 Note that $\cF_{44}$ is not unirational since its Kodaira dimension is non-negative by \cite{GHS07}.
\end{remark}

\begin{remark}
 Let $S\in \cF^{\Lambda_{12}^{0,-2}}$ be a $K3$ surface. Let $H$ and $C$ be the classes in $\Pic(S)$ with $H^2=22$, $H.C=0$ and $C^2=-2$. The polarization of genus $44$ is given by $L = 2H-C$. Then $S\subset \PP^{44}$ is not Brill--Noether general, since $h^0(L)<h^0(H)\cdot h^0(H-C)$. 
\end{remark}

\newpage

\section{Explicit construction of general K3 surfaces of genus 11}\label{K3genus11}

In \cite{HoffSta}, the explicit equations of a $K3$ surface of genus $11$ containing a conic was the starting point to construct a new family of rational Gushel--Mukai fourfolds. A detailed study of the geometric connection of this family of Gushel--Mukai fourfolds and their associated $K3$ surfaces allows us to present an algorithm to compute the general $K3$ surface of genus $11$. In \cite{RS3} it is presented a construction of associated $K3$ surfaces to cubic fourfolds. We apply a similar method to our family of Gushel--Mukai fourfolds. 

\begin{theorem}
 The moduli space $\cF_{11}$ of $K3$ surfaces of genus $11$ is explicitly unirational. 
\end{theorem}

In particular, we have an algorithm to write down the equations of the general $K3$ surface of genus $11$. 
Already in \cite{Mukg11}, Mukai showed the unirationality of $\cF_{11}$. The unirationality of the moduli space of $n$-pointed $K3$ surfaces of genus $11$ is shown by \cite{FV18} for $n=1$ and by \cite{Ba18} for $n\le 6$. 
Mukai's indirect proof of the unirationality is based on a non explicit birational embedding of the moduli space $\mathcal M_{11}$ of curves of genus $11$ in the Mukai correspondence $\mathcal{P}_{11}$ over $\cF_{11}$.   

We explain our precedure and therefore, recall some results from \cite{HoffSta}. 
Let $\YY=\mathbb{G}(1,4)\cap\PP^8\subset\PP^8$ be a smooth hyperplane section of 
the Grassmannian $\mathbb{G}(1,4)\subset\PP^9$ 
parameterizing lines in $\PP^4$.
In \cite{RS3,HoffSta}, it has been shown that there exists 
an irreducible, unirational, $25$-dimensional family 
$\mathcal{S}\subset\mathrm{Hilb}_{\YY}$ 
of rational surfaces of degree $9$ and sectional genus $2$ 
whose class in $\mathbb{G}(1,4)$ is given by $6\sigma_{3,1}+3\sigma_{2,2}$.
The family $\mathcal{S}$ is explicitly described by an algorithm which is implemented in  
the \emph{Macaulay2} package \emph{SpecialFanoFourfolds} \cite{SpecialFanoFourfoldsSource}. In particular,
we are able to get equations for a surface $S\subset \YY$
corresponding to a general point $[S]\in \mathcal S$.

The Zariski-closure of the family of quadratic sections 
of $\YY$, also known as \emph{(ordinary) Gushel-Mukai fourfolds}, that contain some surface of the family $\mathcal S$
describes a hypersurface in the 
$39$-dimensional projective space $\PP(H^0(\mathcal{O}_\YY(2)))$
of all quadratic sections of $\YY$.
By passing to the quotient modulo
the action of $\mathrm{Aut} \YY$,
this hypersurface gives rise to a Noether--Lefschetz divisor $\mathcal{M}_{20}$
in the $24$-dimensional moduli space $\mathcal{M}$ of Gushel-Mukai fourfolds;
see \cite[Theorem~3.3]{HoffSta}, and see \cite{DIM}
for generalities on Gushel-Mukai fourfolds.

Let $S\subset \YY$ be a general surface of the family $\mathcal S$.
Then $S$ admits inside $\YY$ 
a congruence of $3$-secant conics, 
that is, through the general point of $\YY$ 
there passes a unique conic contained in $\YY$ and that 
cuts $S$ at three points. Furthermore,
the linear system of cubic hypersurfaces in $\YY$ 
with double points along $S$ 
gives a dominant rational map 
\[
 \mu:\YY\dashrightarrow \PP^4
\]
such that its general fibers 
are the conic curves of the congruence to $S\subset \YY$; see \cite{HoffSta}.

Let $X\subset\PP^8$ be a general quadratic section of $\YY$ through $S$,
hence giving rise to a general Gushel-Mukai fourfold $[X]\in\mathcal{M}_{20}$.
The restriction of $\mu$ to $X$ 
gives a birational map $\mu|_X:X\dashrightarrow\PP^4$,
and from this we deduced in \cite[Theorem~3.4]{HoffSta} 
that the Gushel-Mukai fourfolds in $\mathcal{M}_{20}$ are rational.

The inverse map 
$$(\mu|_X)^{-1}:\PP^4\dashrightarrow X$$
is defined by the linear system 
$|H^0(\mathcal{I}_{U,\PP^4}^2(9))|$
of hypersurfaces of degree $9$ 
with double points along a certain irreducible surface $U\subset\PP^4$.
It turns out that 
$U$ can be realized as  a triple projection
followed by a simple projection of a minimal K3 surface $\tilde{U}\subset\PP^{11}$ of degree 
$20$ and genus $11$.
In particular,
$U$ has four apparent double points and contains an \emph{exceptional} line $L$ and an
\emph{exceptional} twisted cubic curve $C$.
Since $\tilde{U}$ 
is
the \emph{associated K3 surface} to a general Gushel-Mukai fourfold 
in $\mathcal{M}_{20}$ (see \cite{DIM} and see also \cite{BrakkePertusi}), we have 
that $\tilde{U}$ gives rise to a general point 
in the moduli space $\mathcal{F}_{11}$ 
of polarized K3 surfaces of genus $11$ (and degree $20$).
Our goal is to determine (explicit) equations for $\tilde{U}\subset\PP^{11}$. 

First of all, we remark that 
if we have equations 
for $U$, $L$, and $C$, 
then using standard methods 
we are able to get a birational map $f:U\dashrightarrow\tilde{U}$,
and hence equations for $\tilde{U}$ with two marked points, that is, the constractions of $L$ and $C$. Indeed, the map $f$ can be roughly defined by the linear system
$|H+3C+L|$, where $H$ denote the hyperplane section class (for this step,
we also need to desingularize $U$ by blowing up its singular locus).
Moreover, if we have equations 
for $S$ and $X$ we can also determine equations 
for $U$ (by computing the base locus of the inverse $\mu|_X$ induced by $S$). Thus we only needs to determine the curves $L$ and $C$ in $U$.

A key remark done in \cite{RS3} in the contest of cubic fourfolds
applies also in our case of Gushel-Mukai fourfolds:
\emph{the surface $U$ depends on the pair $(S,X)$,
but its exceptional curves $L$ and $C$ only depend on $S$ but not on $X$.}
In fact, the locus $L\cup C\subset\PP^4$ can be also described 
as the closure of the locus of points 
$p\in\PP^{4}$ such that the fiber $\overline{\mu^{-1}(p)}$ has 
dimension at least $2$.
In practice, we can obtain the curves $L$ and $C$ 
as the top components 
of the intersection $U\cap U'$,
where $U'\subset\PP^4$ is the 
corresponding surface 
to another 
general quadratic section $X'$ of $\YY$ through $S$,
that is, $U'$ is the surface that determines the 
inverse map of the restriction of $\mu$ to $X'$.

In summary, let 
$$
I=\{(S,X)\ | \ S\subset X\}\subset \mathcal{S}\times \PP(H^0(\mathcal{O}_{\YY}(2))) /\mathrm{PGL}(9)
$$
be the $23$-dimensional unirational incidence variety of pairs consisting of surfaces of degree $9$ and sectional genus $2$ and a Gushel--Mukai fourfold $X$ (note that $\mathrm{PGL}(9)$ acts diagonally). 
We get a rational map $\varphi: I \dashrightarrow \cF_{11}$ which maps a pair $(S,X)$ to a $K3$ surface of genus $11$. 

We have implemented this construction in our \emph{Macaulay2} package \emph{K3s}; 
see also the function \emph{associatedK3surface} in the package \emph{SpecialFanoFourfolds}.

\section{The unirationality of the moduli space of K3 surfaces - an overview}\label{knownResults}

We collect unirationality results concerning the moduli space $\cF_g$ of (quasi)-polarized $K3$ surfaces of genus $g$ in the following table.

The Kodaira dimension $\kappa(\cF_g)$ is non-negative for $g\in \{41,43,44,47\}$ or $g\ge 49$, and $\cF_g$ is of general type for $g\in \{47,51,55,58,59,61\}$ or $g\ge 63$ by \cite{GHS07}.
Our goal is to find an implementation of the remaining unirationality constructions mentioned in the above table (that is, $g\in \{13, 16, 18, 20\}$). 

The rationality of the universal $K3$ surface over $\cF_8$ is shown in \cite{DiTullio}.

\begin{center}
\begin{tabular}{|c|l|l|}
\hline
genus $g$ & paper & method of proof \\ \hline \hline
$\begin{matrix} 2\\ \vdots  \\ 5 \end{matrix}$ & classicaly known & complete intersection in projective space \\ \hline 
$\begin{matrix} 6\\ \vdots  \end{matrix}$ & \cite{Muk88} & complete intersection w.r.t. vector bundles \\ 
10 & & on homogeneous spaces \\ \hline
11& \cite{Mukg11} & higher rank Brill--Noether theory on curves\\ \hline
12& \cite{mukai-biregularclassification}, \cite{MukaiFano3folds}, \cite{Mukg13} & c.i.  w.r.t. homogeneous vector bundles \\ \hline
13& \cite{Mukg13} & c.i.  w.r.t. homogeneous vector bundles  \\ \hline
   & \cite{Nuer},  & connection to cubic fourfolds $\mathcal C_{26}$\\ 
14 & \cite{FV18} & $\cF_{14,1}$ is birational to a $\PP^{12}$-bundle over the \\ 
 & & Hilbert scheme of $3$-nodal septic scrolls in $\PP^5$ \\ \hline 
15& not known \\ \hline
 & & c.i. w.r.t. an almost homogeneous \\ 
16 & \cite{Mukg16} & vector bundle of rank 10 in a certain  \\ 
 & & compactified moduli space of twisted cubics \\ \hline
17& announced in \cite{Mukg17} & higher rank Brill--Noether theory on curves \\ \hline
18& \cite{Mukg1820} & c.i. w.r.t. homogeneous vector bundles \\ \hline
19& not known \\ \hline
20& \cite{Mukg1820} & c.i. w.r.t. homogeneous vector bundles \\ \hline
  & \ & connection to cubic fourfolds $\mathcal C_{42}$\\ 
22& \cite{FV20} & $\cF_{22,1}$ is birational to a $\PP^{5}$-bundle over the \\ 
 & & Hilbert scheme of $8$-nodal nonic scrolls in $\PP^5$ \\ \hline 
\end{tabular}
\end{center}

\section{The \emph{Macaulay2} package \emph{K3s}}\label{comput}

In this section, we introduce the \emph{Macaulay2}
 package \emph{K3s} \cite{K3sSource} by referring to its documentation for more details.
 
The main function of the  package is named \texttt{K3}. 
It can be used in two different ways.

\paragraph{First way of using \texttt{K3}}  The function accepts as input a sequence $(d,g,n)$  of three integers 
(as the ones in Table~\ref{unirationalCases})
and
a coefficient ring $K$ (this last input is optional and a particular large finite field is used by default). Then it constructs a random K3 surface $S\subset\PP^g_{K}$
of genus $g$ that contains a curve $C$ of degree $d$ and self-intersection $n$,
as described previously in this paper. Hence, the curve
$C$ together with the hyperplane section class $L$ of $S$ 
span a
rank $2$ lattice $\Lambda^{d,n}_{g}$ with the intersection matrix 
$$ 
\begin{pmatrix}2g-2 & d  \\ d & n \end{pmatrix}.
$$
In this case, 
the output object looks like a function, which takes as input
a pair of integers $(a,b)$ and returns the image 
of the embedding of $S$ defined by the complete linear system $|a L + b C|$
(an error is through if $a L + b C$ is not very ample).
As one specific example, the command \texttt{(K3(5,9,0))(1,1)}
returns a K3 surface of genus $14$
and degree $26$ in $\PP^{14}$ cut out by $66$ quadric hypersurfaces and which
contains an elliptic curve of degree $9$; see also Subsection~\ref{K3genus14}.
{\footnotesize 
\begin{verbatim}
$ M2 --no-preload
i1 : needsPackage "K3s";
i2 : S = K3(5,9,0)
o2 = K3 surface with rank 2 lattice defined by the intersection matrix: | 8 9 |
                                                                        | 9 0 |
o2 : Lattice-polarized K3 surface
i3 : S(1,1)
o3 = K3 surface of genus 14 and degree 26 in PP^14
o3 : Embedded K3 surface
i4 : degrees oo -- degrees of the generators of the ideal
o4 = {({2}, 66)}
\end{verbatim}
}
\noindent
As another example, the command \texttt{(K3(11,2,-2))(2,1)} returns
a  
K3 surface of genus $44$
and degree $86$ in $\PP^{44}$ cut out by $861$ quadric hypersurfaces and 
which contains an irreducible conic; see also Subsection~\ref{K3genus44}.
{\footnotesize 
\begin{verbatim}
i5 : T = K3(11,2,-2)
o5 = K3 surface with rank 2 lattice defined by the intersection matrix: | 20  2 |
                                                                        | 2  -2 |
o5 : Lattice-polarized K3 surface
i6 : T(2,1)
o6 = K3 surface of genus 44 and degree 86 in PP^44
o6 : Embedded K3 surface
i7 : degrees oo -- degrees of the generators of the ideal
o7 = {({2}, 861)}
\end{verbatim}
}
\noindent 
On a standard laptop computer, the whole computation for the first example takes less than $2$ seconds,
while for the second example takes about $3$ minutes and half.

\paragraph{Second way of using \texttt{K3}} 
If the input is just an integer $g$ (with $3\leq g \leq 12$) and a coefficient ring $K$ (optional as before),
then  a random K3 surface of genus $g$ and degree $2g-2$ in $\PP^g_{K}$ 
is returned. 
The most relevant case is $g=11$ 
where the procedure described in Section~\ref{K3genus11} is performed, taking about  $7$ minutes.
One can enable an option \texttt{Verbose} to display more details on the steps of the computation.
{\footnotesize
\begin{verbatim}
i8 : K3 11
o8 = K3 surface of genus 11 and degree 20 in PP^11
o8 : Embedded K3 surface
\end{verbatim}
}

\begin{remark}
All of our constructions also work over $\QQ$. 
\end{remark}


\end{document}